\documentclass[
                 color-hyperref,               
                 graphics,                     
                 redefine-eqnarray,            
                 eucal,                        
                 doublestroke,     
                 upright,                      
               amsmath,amssymb
               ]{amsart}

\usepackage{float}
\usepackage{amssymb}
\title[Ideal theory of intersections of prime divisors]{The ideal theory of intersections of prime divisors dominating a  normal Noetherian local domain of dimension two}

\author[W. Heinzer]{William Heinzer}
\address{Department of Mathematics, Purdue University, West
Lafayette, Indiana 47907}

\author[B. Olberding]{Bruce Olberding}
\address{Department of Mathematical Sciences, New Mexico State University,
Las Cruces, NM 88003-8001}

\theoremstyle{plain}
\newtheorem{theorem}{Theorem}[section]
\newtheorem{corollary}[theorem]{Corollary}
\newtheorem{lemma}[theorem]{Lemma}
\newtheorem{proposition}[theorem]{Proposition}

\theoremstyle{definition}

\newtheorem{notation}[theorem]{Notation}

\theoremstyle{remark}
\newtheorem{remark} [theorem]{Remark}

\usepackage{amsmath}
\usepackage{url}

\def\ff{\frak}
\def\Spec{\mbox{\rm Spec}}

\def\Max{\mbox{\rm Max}}

\def\Zar{\mbox{\rm Zar}}

\def\cal{\mathcal}

\def\X{{\rm Zar}}

\begin{document}

\begin{abstract} Let $R$ be a normal Noetherian local domain of Krull dimension two. We examine intersections of rank one discrete valuation rings  that birationally dominate $R$. We restrict to  the class of prime divisors that dominate $R$ and show that if a collection of such prime divisors is taken below a certain ``level,'' then the intersection is an almost Dedekind domain having the property that every nonzero ideal can be represented uniquely as an irredundant intersection of powers of maximal ideals.

\end{abstract}

\maketitle






\section{Introduction} Throughout this article, $R$ denotes a normal Noetherian local domain of dimension two, and $F$ denotes its quotient field.  
A valuation overring $V$ of $R$ whose maximal ideal contains the maximal ideal of $R$ and whose residue field has transcendence degree 1 over the residue field of  $R$ is a {\it prime divisor that dominates $R$}.\footnote{In the setting of two-dimensional regular local rings, prime divisors that dominate $R$ are  called {\it prime divisors of the second kind} in \cite{ZS}.} 
  The prime divisors that dominate $R$ are precisely the  overrings of $R$ that  arise as the localization of the integral closure of a finitely generated $R$-subalgebra of $F$ at a height one prime ideal that contains the maximal ideal of $R$ \cite{SH}.  
  Prime divisors 
  are familiar objects in the birational algebra of Noetherian local domains. They 
  play an important role in embedded resolution of singularities \cite{Abh2}, function fields of surfaces \cite{Spiv}, and in the theory of integral closure of ideals, where they appear as Rees valuations of ideals \cite{SH}.
  
  Our interest in this article is in the intersection of prime divisors that dominate $R$.  The intersection of finitely many such rings is a PID \cite[(11.11), p.~98]{N}, while 
  the intersection of  all prime divisors that dominate $R$ is simply $R$ and so is quite far from being a PID. Other examples that illustrate the wide range of possibilities for these  intersections  can be found in \cite{HLO, HO2}. In this article we focus on the intersection of prime divisors that dominate $R$ and 
  are within some fixed number of steps away from $R$.  The steps here involve the number of normalized local quadratic transforms needed to reach the prime divisor; see Section 2.  We prove that the intersection in this case is an {\it almost Dedekind domain,} meaning that each localization at a maximal ideal is a DVR (i.e., a rank one discrete valuation ring).  Unlike the situation in Dedekind domains, ideals in an almost Dedekind domain need not factor into a finite product of ``nice'' ideals such as prime or  radical ideals, and in general the factorization theory of almost Dedekind domains is complicated; see for example \cite{HOR} and \cite{LL}  and their references for more on this. 
 
 For commutative rings, the absence of  ideal factorizations  can sometimes  be remedied by intersection decompositions of ideals, as is the case with primary decomposition in Noetherian rings.  It is not hard to see that every ideal in an almost Dedekind domain is an intersection of primary ideals, but because this intersection is typically infinite and may not be able to be refined to an irredundant intersection, this decomposition is too unwieldy to be useful for arbitrary almost Dedekind domains. In Theorem~\ref{qs Prufer} we show that our almost Dedekind domains  are more special in that they admit just such a 
  decomposition: Every nonzero ideal $I$ can be represented uniquely as an irredundant intersection of primary ideals. The primary ideals in an almost Dedekind domain are simply the powers of the maximal ideals, and so the decomposition can be restated for powers of maximal ideals. The powers of maximal ideals, in turn, are precisely
 the nonzero {completely irreducible} ideals of an almost Dedekind domain, 
 where an ideal is {\it completely irreducible} if it is not the intersection of any set of 
 proper overideals.
 
 Our search for a decomposition theory for  rings obtained as an  intersection of prime divisors  is motivated by our work with Laszlo Fuchs in the series of papers \cite{FHO,FHO2,FHO3, FHO4}, where among several other ideal decompositions  we studied those involving completely irreducible ideals. In the article \cite{HO} we characterized the commutative rings for  
which every ideal can be represented uniquely as an irredundant intersection of completely irreducible ideals. The intersections of prime divisors considered in the present article thus provide a new example of a class of rings having such an intersection decomposition.  

We thank Laszlo for introducing us to a number of these topics, first  
 through his seminal article \cite{F} and then through our  work with him on ideal theory. We especially thank him for sharing with us his remarkable and elegant way of thinking about mathematics, as well as his gift for  exposition   
   that all those familiar with his work recognize.


%
%


\section{Normal sequences}

Recall our standing hypothesis that $R$ is a normal Noetherian local domain of Krull dimension $2$ with quotient field $F$. 
For each local overring $S$ of $R$, we let ${\ff X}(S)$ denote the set of valuation overrings $V$ of $S$ that {\it dominate} $S$, meaning that the maximal ideal of $S$ is contained in the maximal ideal of $V$.   
 We show in this section that each valuation ring in ${\ff X}(R)$ is specified by a sequence of ``points'' which   lie on normalized blowups of $\Spec(R)$.  
 
To formalize this,
let ${\ff m}$ denote the maximal ideal of $R$, and let 
 $x_1,\ldots,x_n \in {\ff m} \setminus {\ff m}^2$ be such 
 that $x_1,\dots,x_n$ generate  ${\ff m}$. 
A  {\it local quadratic transform} of $R$ is a ring of the form 
$R_1 =  R[x_1/x_i,\ldots,x_n/x_i]_P$, where   $x_i \not \in {\ff m}^2$ 
and $P$ is a prime ideal of 
$ R[x_1/x_i,\ldots,x_n/x_i] $ that contains ${\ff m}$.  
Let $\overline{R_1}$ denote the integral closure of $R_1$ in $F$.
By the Krull-Akizuki Theorem  \cite[Theorem 33.2, p.115]{N}, $\overline{R_1}$ is a Noetherian domain. 
  If $M$ is a maximal ideal of $\overline{R_1}$ containing ${\ff m}$, we say that 
  the Noetherian local domain $(\overline{R_1})_M$ is a {\it normalized quadratic transform of $R$}. 
  If $\{R_i\}$ is a (finite or infinite) sequence of local overrings of $R$ such that $R = R_0$ and $R_{i+1}$ is a normalized quadratic transform of $R_i$ for each~$i$, then, following 
  Zariski \cite[p.681]{Z} and            Lipman \cite[p. 201]{LipRat}, we say  $\{R_i\}$ is a {\it normal sequence} over $R$.  
  Motivated by the terminology in \cite{Lip}, we say that a local ring $S$ of Krull dimension 2 that occurs in some normal sequence over $R$ is a  {\it point} of $R$. If  $R$ is a two-dimensional  regular local ring, then the regular local overrings of $R$ are
  precisely the points of $R$ \cite[Theorem 3]{Ab}.
  
    These valuation rings in ${\ff X}(R)$ will be the main focus of this section.
     We denote by ${\rm Div}(R)$ the subset of ${\ff X}(R)$ that consists  of the prime divisors that dominate $R$. In the next lemma we note  that the prime divisors that dominate $R$ occur as endpoints of finite normal sequences. This follows from the more general properties of normal sequences collected in the following lemma. 
 Statement (1) is due to Abhyankar \cite[Lemma 12]{Ab} in the case in which $R$ is a two-dimensional regular local ring.  Lipman points out in \cite[p. 202]{LipRat}  that  Abhyankar's proof can be adjusted to the more general setting of infinite normal sequences by
 piecing together the following information in  \cite{ZS}:  Corollary, p.~21;  Proposition 1, p.~330;  Corollary 2, p.~339; and the argument in the middle of p.~392.
  The other statements of the lemma are implicit in the literature of quadratic transforms of regular local rings. We include a few details for lack of a precise reference for our setting.

\begin{proposition} \label{Lipman lemma} $\:$
\begin{itemize}

\item[$(1)$] 
If $\{R_i\}$ is an infinite normal sequence over $R$, then $\bigcup_{i}R_i$ is a valuation ring in ${\ff X}(R)$.  

\item[$(2)$]  If $V \in {\ff X}(R)$, then  there is a unique normal sequence  $\{R_i\}_{i=0}^d$, with $d$ possibly infinite, such that $V = \bigcup_{i=0}^d R_i$. 

\item[$(3)$]  If $S$ is a point of $R$, then there is  a unique normal sequence $\{R_i\}_{i=0}^d$, with $d$ finite, such that $R_d = S$.

\item[$(4)$] 
A valuation ring $V \in {\ff X}(R)$ is  a prime divisor that dominates $R$ if and only if  the normal sequence  along $V$ is finite.

\end{itemize}
\end{proposition}

\begin{proof} (1) See \cite[p.~202]{LipRat} and the references given there to  \cite{ZS}. 

(2) 
The sequence $\{R_i\}$ is constructed as follows. Since $V$ is a valuation ring that dominates $R$, there is an  $i$ such that $x_iV  =  (x_1,\dots,x_n)V$. 
Thus $S:=R[x_1/x_i,\ldots,x_n/x_i] \subseteq V$.  Let $P$ be the center of $V$ in $\overline{S}$. Then $R_1 = \overline{S}_P$ is a normalized quadratic transform of $R$ that is dominated by $R_1$. If $R_1$ has Krull dimension $1$, then $R_1$ is a DVR dominated by $V$, which forces $R_1 = V$. In this case, we set $d = 1$ and the sequence in (2) is obtained. Otherwise, $R_1$ has Krull dimension $2$, and we repeat the process. In this way we obtain a normal sequence $\{R_i\}_{i=0}^\infty$, possibly infinite, for which $V$ dominates $\bigcup_{i =0}^dR_i$.  If $d$ is finite, then, as we have already observed, $V$ is a DVR with $V = \bigcup_{i =0}^dR_i$. If $d$ is infinite, then, by (1), $\bigcup_{i=0}^\infty R_i$ is a valuation ring dominated by $V$, which forces $\bigcup_i R_i = V$.  
That   $\{R_i\}_{i=0}^d$ is the unique such sequence follows inductively from the following two facts: (a) $R_1$ lies in the normalization of the projective model of $F/R$ defined by $x_1,\ldots,x_n$, and (b) every valuation ring in ${\ff X}(R)$ dominates a unique point in a normalized projective model of $F/R$; see \cite[pp.~119--120]{ZS}.

(3) Let $S$ be a point of $R$.  Then there is a normal sequence $\{R_i\}_{i=0}^d$ such that $R_d = S$. Let $V$ be a valuation ring in ${\ff X}(R)$ that dominates $S$. By~(2), $\{R_i\}_{i=0}^d$ forms part of a unique normal sequence whose union is $V$, so statement (3)  follows. 

(4)  Suppose $V$ is  a prime divisor that dominates $R$. By (2), there is a normal sequence $\{R_i\}_{i=0}^d$ such that $V = \bigcup_{i = 0}^dR_i$.  If $\{R_i\}$ is infinite, then each $R_i$ has Krull dimension~2, so that by the Dimension Inequality \cite[Theorem 15.5, p.~118]{Mat} the residue field of $R_i$ is algebraic over that of $R$. But then the residue field of $V$ is algebraic over the residue field of $R$, contrary to the fact that $V$ is a prime divisor that dominates $R$. Thus $\{R_i\}$ is finite. 

Conversely, if $\{R_i\}_{i=0}^d$ is finite, then $V$ is a normalized quadratic transform of the two-dimensional integrally closed Noetherian local domain $R_{d-1}$. As such, $V$ is a DVR that is a localization of a nonmaximal prime ideal of a two-dimensional local Noetherian domain, and hence the residue field of $V$ is not algebraic over the residue field of $R$. Thus $V$ is a prime divisor that dominates $R$.          
\end{proof}

In light of Proposition~\ref{Lipman lemma}(3) and (4), the set of local overrings of $R$ that appear in a normal sequence are precisely the points of $R$ and the prime divisors that dominate $R$. 
If $S$ is a point of $R$ where 
$R = R_0 \subseteq R_1 \subseteq \cdots \subseteq R_d=S$ is the unique normal sequence that terminates at $S$, we say that the {\it level} of $S$ is $d$.  
Similarly, if $V$ is  a prime divisor that dominates $R$ and $\{R_i\}_{i=0}^d$ is the normal sequence that terminates at $V$,  the {\it level} of $V$ is $d$.


\begin{remark} \label{point remark} {Let $d>0$. Each point $S$ of $R$ of level $d$ is dominated by only finitely  many prime divisors that dominate $R$ of level $d+1$. These prime divisors that dominate $R$ correspond to the Rees valuation rings of the maximal ideal of $S$; see \cite{SH}.    In particular, there are only finitely many prime divisors that dominate $R$ at level 1.
}
\end{remark}

\section{Normal sequences and limit points}
 
 Let $S$ be a local overring of the two-dimensional integrally closed local Noetherian domain $R$.  
We view ${\ff X}(S)$ as a subspace of the Zariski-Riemann space $\Zar(S)$ of valuation rings  of $F$ that contain $S$. The topology on $\Zar(S)$ has  as a basis of open sets the subsets of $\X(S)$ of the form \begin{center}
${\cal U}(t_1,\ldots,t_n):=\{V \in \X(S):t_1,\ldots,t_n \in V\},$ where $ t_1,\ldots,t_n \in F.$
\end{center}
With this topology, $\X(S)$ is a quasicompact $T_1$ space. (It is in fact a spectral space; see for example \cite{DF}, \cite{HK} and \cite{OGraz}.) There is a helpful refinement of the Zariski topology to a Hausdorff topology that is obtained by taking as an open basis of $\X(S)$ the sets of the form $U \cup V$, 
 where $U$ is open and quasicompact in the Zariski topology 
and  $V$ is the complement of a quasicompact open set.
The resulting topology is the {\it patch} (or {\it constructible}) topology on $\X(S)$. For more background on the patch topology for the Zariski-Riemann space of a field, see \cite{FFL} and \cite{OZR} and their references. 
 This topology is quasi-compact, Hausdorff and has a basis of clopen ($=$ closed and open) sets. 
   It is in this topology that we work because of its suitability for dealing with limit points.

\begin{notation} For $X \subseteq \X(R)$, we let
%
\begin{center} $\lim(X) = $ the set of patch limit points of $X$ in $\X(R)$. 
\end{center}
%


\end{notation}

In Lemma~\ref{lim check} and Proposition~\ref{div theorem} we develop some relationships between limit points and normal sequences.  The following lemma, which generalizes \cite[Lemma 4.5]{HLO}, is useful as a means of specifying open neighborhoods of limit points.   

\begin{lemma} \label{closed point} If $S$ is a point of $R$, then 
 ${\ff X}(S)$ is patch clopen in $\X(R)$. 
\end{lemma} 

\begin{proof} Let $d$ be the level of $S$.   
Since $S$ is a point, there is a normal sequence $\{R_i\}_{i=1}^d$ with $R_d = S$. Thus $${\ff X}(S) = {\ff X}(R_d) \subseteq {\ff X}(R_{d-1}) \subseteq \cdots \subseteq  {\ff X}(R_0) = {\ff X}(R).$$   
We show that ${\ff X}(S)$ is patch clopen in $\X(R)$ by proving that 
 ${\ff X}(R_{i+1})$ is patch clopen   in $\X(R_i)$   for each $0 \leq i < d$.

Let $0 \leq i < d$. 
There exist  $t_1,\dots,t_n \in F$ such that for $A = \overline{R_i[t_1,\ldots,t_n]}$, we have $R_{i+1} = A_{P}$ for some prime ideal $P$ of $A$. Since $S$ has Krull dimension $2$, so does $R_{i+1}$; thus $P$ is a maximal ideal of $A$. 
Let  ${\cal U}:={\cal U}(t_1,\ldots,t_n) \cap \X(R_i)$.   Then ${\ff X}(R_{i+1}) \subseteq {\cal U} \subseteq \X(R_i)$ and ${\cal U} $ is 
 patch clopen in ${\X}(R_i)$, so to show that ${\ff X}(R_{i+1})$ is patch clopen in ${\X}(R_i)$, it suffices to show that  ${\ff X}(R_{i+1})$ is patch clopen in ${\cal U}$.
 
Since $R_{i+1}$ is a Noetherian ring, $\Spec(R_{i+1})$ is a Noetherian space. Thus, since $\{PR_{i+1}\}$ is a closed set in $\Spec(R_{i+1})$, we have that the set $\Spec(R_{i+1}) \setminus \{PR_{i+1}\}$ is    open and    quasicompact in $\Spec(R_{i+1})$. Therefore, $\{PR_{i+1}\}$ is patch clopen in $\Spec(R_{i+1})$. 
The   map $\delta:{\cal U} \rightarrow \Spec(R_{i+1})$ that sends a valuation ring in $\Zar(R_{i+1})$ to its center in $R_{i+1}$  is continuous in the patch topology \cite[Proposition~4.2]{OZR}.
 Thus, since ${\ff X}(R_{i+1}) = \delta^{-1}(PR_{i+1})$, the set  ${\ff X}(R_{i+1})$ is patch clopen in ${\cal U}$, which completes the proof that $
 X(S)$ is patch clopen in ${\X}(R)$. \end{proof} 




\begin{lemma} \label{lim check} Let $U \in {\ff X}(R)$ such that $U$ is not a prime divisor, let $\{R_i\}_{i=0}^\infty$ be the infinite normal sequence along $U$, and let $ X$ be a nonempty subset of ${\ff X}(R)$. Then \begin{center} $U \in \lim(X) \:\:\Longleftrightarrow \: \: {\ff X}(R_i) \cap (X \setminus \{U\}) \ne \emptyset$ for all $i \geq 0$.\end{center}
\end{lemma} 

\begin{proof} Let $X' = X \setminus \{U\}$.   
By Lemma~\ref{closed point}, each ${\ff X}(R_i)$ is a patch clopen subset of $\X(R)$. Also, by Proposition~\ref{Lipman lemma}(1), $\bigcap_{i\geq 0}{\cal {\ff X}(R_i)} = \{U\}$.  If $U \in \lim(X)$ and $i \geq 0$, then, since  ${\ff X}(R_i)$ is a patch open neighborhood of $U$, the set   $X'   \cap {\ff X}(R_i)$ is nonempty. 
Conversely, suppose that for each $i\geq 0$ the set ${\ff X}(R_i) \cap X'$ is nonempty. 
 Then $\{{\ff X}(R_i)\}_{i=0}^\infty$ is a descending sequence of patch closed subsets of ${\ff X}(R)$, each of which  intersects   $X'$.   Since the patch closure $\overline{X'}$ of $X'$ in $\X(R)$ is patch quasicompact, the set $\overline{X'} \cap (\bigcap_i {\ff X}(R_i)) = \overline{X'} \cap \{U\}$ is nonempty; i.e., $U \in \overline{X'}$. Since $U \not \in  X'$, we conclude that 
%
$U \in \lim(X)$.  
\end{proof}

The following lemma, which is one of the main theorems in \cite{OCompact}, will be the means by which we establish that the intersection of prime divisors of bounded level that dominate $R$ is an almost Dedekind domain. Let $D$ be an integral domain.  For a subset $X $ of $\X(D)$, we let 
  $$A(X) = \bigcap_{V \in X}V \: {\mbox{ \rm and }} \:  J(X) = \bigcap_{V \in X}{\ff M}_V,$$ where for each valuation ring $V$, ${\ff M}_V$ is the maximal ideal  of $V$.  

   \begin{lemma}
 \label{OCompact}  
    \label{rank one theorem}  \cite[Theorem 5.3 and Corollary 5.4]{OCompact} Let $D$ be a domain with quotient field $Q(D)$. The following are equivalent for a nonempty subset $ X \subseteq \X(D)$  with $J(X) \ne 0$.    
       
      \begin{itemize}
      \item[{\rm (1)}] $A(X)$  is a one-dimensional Pr\"ufer domain with   quotient field $Q(D)$. 
      \item[{\rm (2)}] 
       $X$ is contained in a    quasicompact set of rank one valuation rings in $\X(D)$.
       
       \item[{\rm (3)}]  Every valuation ring in the patch closure of $X$ has rank one. 
       \end{itemize} If also $X$ is contained in a quasicompact set of DVRs, then $A(X)$ is an almost Dedekind domain. 
  \end{lemma} 

Ultimately we want to apply the lemma to sets of prime divisors that dominate $R$, but we show first in the next proposition that the lemma has a more general application to rank one valuation rings.  

\begin{proposition} \label{div theorem} 
Let $X$ be a nonempty set of rank one valuation rings in ${\ff X}(R)$.    
The following statements are equivalent and are sufficient for $A(X)$ to be a one-dimensional Pr\"ufer domain with nonzero Jacobson radical.     
\begin{itemize}
\item[{\rm (1)}] Each valuation ring in $\lim(X)$ is a prime divisor that dominates $R$.

\item[{\rm (2)}]  For each infinite normal sequence $\{R_i\}$ there is $i \geq 0$ such that  $$|\X(R_i) \cap X| \leq 1.$$


\item[{\rm (3)}] 

For each infinite normal sequence $\{R_i\}$ either there is $i \geq 0$ such that $ \X(R_i) \cap X = \emptyset$ or 
 the valuation ring $\bigcup_{i=0}^\infty R_i$  is  a patch isolated point in $X$.

 \end{itemize}
 
\end{proposition}

\begin{proof} Lemma~\ref{rank one theorem}
implies that  (1) is sufficient for $A(X)$ to be a one-dimensional Pr\"ufer domain with nonzero Jacobson radical. It remains to show the equivalence of (1)--(3). 

(1) $\Rightarrow$ (2)
 Let  $\{R_i\}$ be an infinite normal sequence, and let $U = \bigcup_{i}R_i$. By Proposition~\ref{Lipman lemma}(1), $U$ is a valuation overring of $R$, and, by Proposition~\ref{Lipman lemma}(4), $U$ is not a prime divisor that dominates $R$. By (1), $U \not \in \lim(X)$, so by Lemma~\ref{lim check} we have   $\X(R_i) \cap (X \setminus \{U\}) = \emptyset$ for  $i \gg 0$. Statement (2) now follows.

(2) $\Rightarrow$ (3) Let $\{R_i\}$ be an infinite normal sequence over $R$, let $V = \bigcup_i R_i$, and suppose that for each  $i \geq 0$, $X \cap \X(R_i) \ne \emptyset$. 
By (2) there is for $i \gg 0$ exactly one valuation ring $U \in X$ that dominates  $R_i$.  Thus $V = \bigcup_{i}R_i \subseteq U$ and since $U$ dominates the valuation ring $V$, it follows that $V = U \in X$.  
To see that  $V \not \in \lim(X)$, choose $i \gg 0$ such that $V$ is the only valuation ring in $X$ that dominates $R_i$.  By Lemma~\ref{closed point}, ${\ff X}(R_i)$ is a patch open neighborhood of $V$. Since this set contains no other members of $X$, we conclude that $V$ is a patch isolated point in $X$.


(3) $\Rightarrow$ (1)  
Let $U \in \lim(X)$.     
We observe first   that $U$ dominates $R$.  
 Indeed, let $0 \ne x $ be in the maximal ideal of $R$, and suppose $x \not \in {\ff M}_U$.   Then $U \in {\cal U}(x^{-1})$, and since $U \in \lim(X)$, there are infinitely many valuation rings $V$ in $X \cap {\cal U}(x^{-1})$. We have $x^{-1} \in V$ for each such valuation ring $V$, so that $x \not \in {\ff M}_V$, contrary to the fact that each $V \in X$ dominates $R$. Thus $U$ dominates $R$.

 Now let $\{R_i\}$ be the normal sequence along $U$.  If $\{R_i\}$ is infinite, then Lemma~\ref{lim check} implies that  $X  \cap \X(R_i) \ne \emptyset$ for all $i$, which, since $U$ is not a patch isolated point in $X$, is contrary to (3). Thus $\{R_i\}$ is finite and so $U$ is a prime divisor. 
 %
%
%
%
\end{proof} 


 




\section{Prime divisors of bounded level}

 For each integer $d \geq 1$, we denote by ${\rm Div}_d(R)$ the set of prime divisors that dominate $R$ and are of level at most~$d$.  For $d=1$, this set is finite and 
  consists of the Rees valuation rings for the maximal ideal of $R$, but for $d >1$, the set ${\rm Div}_d(R)$ is infinite.  In this section we are interested in the rings obtained as  intersections of valuation rings from the set ${\rm Div}_d(R)$. In Theorem~\ref{qs Prufer}  these rings are shown to be almost Dedekind domains by proving that ${\rm Div}_d(R)$ is patch closed and applying Proposition~\ref{div theorem}.

  This sequence of ideas requires 
more information about the patch limit points of this subset of $\X(R)$, and for this purpose we recall the notion of the Cantor-Bendixson derivative of a topological space $X$: Define $X^0 = X$, and for each ordinal number $\alpha$ let $X^{\alpha+1}$ denote the set of limit points of $X^\alpha$.   For each limit ordinal $\lambda$, let $X^{\lambda} = \bigcap_{\alpha< \lambda}X^\alpha$. The closed set $X^\alpha$ of $X$ is the $\alpha$-th {\it Cantor-Bendixson derivative} of $X$. The {\it Cantor-Bendixson rank} of $X$ is the smallest ordinal number $\alpha$ for which $X^{\alpha} = X^{\alpha + 1}$.  


\begin{lemma} \label{count} \label{CB} Let 
   $d \geq 2$, and let $X$ be a nonempty set of rank one valuation rings in ${\ff X}(R)$ such that each point of $R$ of level $d$ is dominated by 
   at most finitely many valuation rings in $X$. Then  
   $X^1 \subseteq   {\rm Div}_{d}(R)$. In particular,
if    $X \subseteq {\rm Div}_d(R)$, 
  then $X^{d} = \emptyset$.

   \end{lemma}

   \begin{proof}  Let $U \in \lim(X)$. As in the proof that (3) implies (1) in  Proposition~\ref{div theorem}, $U$ dominates $R$.  
  Let $\{R_i\}_{i=0}^n$ be the normal sequence along $U$, where $n \in {\mathbb N} \cup \{\infty\}$.  To see that $n < d$, suppose to the contrary that $n\geq d$.  
  By Lemma~\ref{closed point}, ${\ff X}(R_d)$  is a patch open neighborhood of $U$ in $\Zar(R)$. By assumption, this neighborhood contains only finitely many valuation rings in $X$, a contradiction to the fact  that $U \in \lim(X)$.   Thus $n < d$. By Proposition~\ref{Lipman lemma}(4),  $U$ is a prime divisor of  level $<d$.  This proves that  $\lim(X) \subseteq   {\rm Div}_{d-1}(R)$. 
  For the last assertion, we have shown that for each $k \geq 1$ we have  $\lim({\rm Div}_{k}(R)) \subseteq   {\rm Div}_{k-1}(R).$ Thus $X^{d-1} \subseteq  {\rm Div}_1(R)$.  Since 
   ${\rm Div}_1(R)$ is finite and nonempty, $X^{d} = \emptyset$.  
 %
 %
%
  %
 %
  %
\end{proof}

Part 1 of the following theorem is proved in \cite[Corollary 4.13]{HLO} in the case in which $R$ is a regular local ring. 

\begin{theorem} \label{qs Prufer}  Let   $d \geq 0$, and let $X$ be a nonempty set of prime divisors that dominate $R$ and occur at level at most $d$. Then
\begin{itemize}
\item[$(1)$]  $A:=A(X)$ is an almost Dedekind domain with nonzero Jacobson radical. 

\item[$(2)$] Every finitely generated ideal of $A$ is principal.

\item[$(3)$]  A maximal ideal $M$ of $A$ is finitely generated if and only if $A_M$ is a patch isolated point in $X$.

\item[$(4)$]  Every nonzero ideal of $A(X)$ can be expressed uniquely as an irredundant intersection of primary ideals (equivalently, powers of maximal ideals).

\end{itemize}

 \end{theorem}

 \begin{proof}
  Lemma~\ref{count} implies that the patch closure of $X$ consists of DVRs, so (1)  follows from Proposition~\ref{div theorem}. 
 By \cite[Theorem 6.1]{OCompact}, every finitely generated ideal in a one-dimensional Pr\"ufer domain with nonzero Jacobson radical is principal, so (2) follows.  
For (3), apply  \cite[Lemma~6.3]{HOR}, and for (4) observe that 
 by  Lemma~\ref{CB}, 
$X^{d+1} = \emptyset$.   Thus, by \cite[Theorem 6.6 and Corollary 6.7]{HOR}, 
every ideal of $A(X)$ is an irredundant intersection of powers of maximal ideals. 
 \end{proof}

   By Theorem~\ref{qs Prufer}  the ring $A(X) = \bigcap_{V \in X}V$ is a one-dimensional Pr\"ufer domain with $J(A) \ne 0$.  For such a ring, it is shown in \cite{HOR} that the  Cantor-Bendixson rank of the maximal spectrum of $A$ reflects aspects of the ideal theory of $A$ related to Loper and Lucas' factorization theory in \cite{LL} for  ideals in almost Dedekind domains.  
 

 \begin{corollary} \label{all}
   Let $d \geq 1$, and  let $X = {\rm Div}_d(R)$. The localizations of $A(X)$ at  maximal ideals are precisely the valuation rings in ${\rm Div}_d(R)$. 
   \end{corollary}
   
   \begin{proof} By Lemma~\ref{count}, $X$ is a patch closed set. By \cite[Proposition~5.6(5)]{OZR}, the fact that $A(X)$ is a Pr\"ufer domain and $X$ is a patch closed set consisting of rank one valuation rings implies that  the set of valuation overrings of $A(X)$ is $X \cup \{F\}$. Since the valuation overrings of a Pr\"ufer domain are precisely the localizations of the domain at its prime ideals, the corollary now follows. 
   \end{proof} 
   
   \begin{remark}  
   Let $A$ be an almost Dedekind domain. It is shown in \cite{HO} that 
$\Max(A)$ is scattered (i.e., every subspace contains an isolated point) if and only if every nonzero proper ideal of $A$ has a unique representation as an irredundant intersection of completely irreducible ideals. 
Thus the maximal spectrum of the ring $A(X)$ in Theorem~\ref{qs Prufer} is scattered, a fact that also follows from direct topological arguments using the observation in Lemma~\ref{count} that the $d$-th Cantor-Bendixsen derivative of $\Max(A(X))$ is the empty set. 
\end{remark}

    It would be interesting to have an intrinsic characterization of the ring $A(X)$ in Corollary~\ref{all} that reflects the fact that the localizations of $A(X)$ at maximal ideals are precisely the prime divisors that dominate $R$ and have  level at most $d$.  If $R$ is 
  regular and  $f$ and $g$ are  nonzero elements of $R$, the rational function $f/g$  has 
  a position at each point in the quadratic tree over $R$ that is described in 
  \cite{HLO} as either a zero, a pole, a unit, or undetermined.  The ring $A(X)$ is an
  overring of $R$ that describes the rational functions that have no poles 
  among  the prime   divisors in  $X = {\rm Div}_d(R)$.



\section{Almost Dedekind domains with specified residue fields} 

The results in this section are similar in spirit to those in the previous sections. They 
 have the related goal of producing almost Dedekind domains that are not Noetherian   in the natural 
 setting of overrings of  a normal Noetherian domain of dimension two.
 

\begin{lemma} \cite[Theorem 5.3]{OTop} \label{5.3} 
Let $A$ be a domain. If  there is a subset $X$ of $\X(A)$  such that $A = A(X)$, $J(X) \ne 0$, $\lim(X)$ is finite, and all but finitely many valuation rings in $X$ have rank one, then  
 A is a Pr\"ufer domain with nonzero Jacobson radical. 
 \end{lemma}


Lemma~\ref{5.3}, 
along with an existence result for PIDs due to Heitmann, gives a way to construct almost Dedekind domains with prescribed residue fields. While  Heitmann \cite{Heit} constructs PIDs with countably many specified countable residue fields, when we use his PID in our context to produce an almost Dedekind domain with specified residue fields  the result is not as tight: Our methods  force us to  accept an additional residue field, one of characteristic $0$, that is not specified in the original construction.   This is  the field $L$ in the next theorem. A residue field of characteristic $0$ cannot be avoided in the construction of a one-dimensional domain with nonzero Jacobson radical and infinitely many maximal ideals. This is because every zero-dimensional domain with infinitely many maximal ideals has a residue field with characteristic~$0$ \cite[2.3, p.~133]{GH}.

\begin{theorem}  \label{Heit} Let ${\cal F}$ be a countable collection of countable fields such that for each characteristic $p \ne 0$, ${\cal F}$ contains only finitely many fields of that characteristic.  Then there is a countable non-Noetherian 
almost Dedekind domain $A$  with quotient field $K$ of  characteristic $0$ such that $A$ has nonzero Jacobson radical and the collection of residue fields of $A$ is precisely ${\cal F} \cup \{L\}$, where  $L$ is a subfield of $K$ such that $K=L(x)$  for some $x \in K$  transcendental over $L$.   
\end{theorem}

\begin{proof}  Heitmann shows in \cite{Heit} that there is a countable  PID $B$ of characteristic $0$ whose collection of residue fields is precisely ${\cal F}$.   Let $T$ be an indeterminate for $B$, and consider $D =B[T]$. 
  For each maximal ideal ${\ff p}=(b)$ of $B$, we choose a DVR centered on the maximal ideal $(b,T)D$ having residue field $D/(b,T)D = B/{\ff p}$. This can be done in the following manner. Since $B_{\ff p}$ is a countable DVR and its completion $\widehat{B}_{\ff p}$ is uncountable, there is an element $\alpha$ of the maximal ideal of $\widehat{B}_{\ff p}$ that is transcendental over $L$, the quotient field of $B$. Then  $$0 = B[T]_{(b,T)} \cap (T-\alpha)\widehat{B}_{\ff p}[[T]],$$ so that $B[T]_{(b,T)}$ embeds into $\widehat{B}_{\ff p}[[T]]/(T-\alpha)$. The latter ring is isomorphic to the DVR $\widehat{B}_{p}$, and identifying $B[T]_{(b,T)}$ with its image in $\widehat{B}_{\ff p}$, this  DVR contracts in $L(T)$ 
   to a DVR centered on $(b,T)D$ with residue field $B/{\ff p}$. 

Now let $X$ be the collection of  all these DVRs, one for each maximal ideal of $B$.  
 We claim that the only patch limit point of $X$ in $\Zar(D)$ is $L[T]_{TL[T]}$, where $L$ is the quotient field of $B$. 
 Suppose that $U$ is a patch limit point of $X$.  
 We show that $U = L[T]_{TL[T]}$.
 By \cite[Corollary 3.8]{FFL},
  there is a nonprincipal ultrafilter ${\cal F}$ on the set $X$ such that   $$U = \{q \in K:\{V \in X:q \in V\} \in {\cal F}\}.$$  
 To show that $L[T]_{TL[T]} \subseteq U$, suppose $g \in D[T] \setminus TD[T]$.  Then since $D[T]$ is a Noetherian ring,  $g$ is contained in only finitely many maximal ideals of $D[T]$ that contain $T$. 
  Thus $1/g \in V$ for all but finitely many $V \in X$, and so $\{V \in X:1/g \in V\} \in {\cal F}$. We conclude that $1/g \in U$, and hence that $L[T]_{TL[T] }= D[T]_{TD[T]} \subseteq U$.   An argument such as in the proof of (3) implies (1) of Proposition~\ref{div theorem} shows that  $T \in {\ff M}_U$. Thus $L[T]_{TL[T]} \subseteq U \subsetneq K$ and since $L[T]_{TL[T]}$
 is a DVR, we have  $L[T]_{TL[T]} = U$.  This proves that $L[T]_{TL[T]} $ is the only patch limit point of $X$.

  Next, by Lemma~\ref{5.3}, $A(X)$ is a one-dimensional Pr\"ufer domain.
  The fact that the patch closure of $X$ is $X \cup \{ L[T]_{(T)} \}$ implies 
   that $A(X)$ is an 
  almost Dedekind domain whose localizations at maximal ideals are 
   the members of this patch closed set; see \cite[Proposition~5.6(5)]{OZR}. The theorem now follows. 
\end{proof} 

In the proof of Theorem~\ref{Heit}, the patch closure of $X$ in $\Zar(D)$ is $X \cup \{L[T]_{(T)}\}$, and so arguments as in the proof of  Theorem~\ref{qs Prufer} show that $A(X)$ has only one maximal ideal that is not finitely generated, and 
 $A(X)$ has the property that every nonzero ideal can be represented uniquely as an irredundant intersection of powers of  maximal ideals.

\end{document}